\documentclass{amsart}
\usepackage{mathptmx, amssymb, colonequals, stmaryrd}
\usepackage{color}
\definecolor{chianti}{rgb}{0.6,0,0}
\definecolor{meretale}{rgb}{0,0,.6}
\usepackage[colorlinks=true, pagebackref, hyperindex, citecolor=meretale, linkcolor=chianti]{hyperref}

\newtheorem{theorem}{Theorem}[section]

\theoremstyle{definition}
\newtheorem{remark}[theorem]{Remark}
\numberwithin{equation}{theorem}

\renewcommand{\ge}{\geqslant}
\renewcommand{\le}{\leqslant}
\renewcommand{\bar}{\overline}

\renewcommand{\to}{\longrightarrow}
\renewcommand{\mapsto}{\longmapsto}
\renewcommand{\mod}{\ \operatorname{mod}\,}

\newcommand{\la}{\langle}
\newcommand{\ra}{\rangle}
\newcommand{\onto}{\relbar\joinrel\twoheadrightarrow}
\newcommand{\frakm}{\mathfrak{m}}

\newcommand{\FF}{\mathbb{F}}
\newcommand{\ZZ}{\mathbb{Z}}

\begin{document}
\title{On an example of Nagarajan}

\author{Annie Giokas}
\address{Department of Mathematics, Purdue University, 150 N University St., West Lafayette, IN~47907, USA}
\email{agiokas@purdue.edu}

\author{Anurag K. Singh}
\address{Department of Mathematics, University of Utah, 155 South 1400 East, Salt Lake City, UT~84112, USA}
\email{singh@math.utah.edu}

\thanks{A.G. was supported by the Undergraduate Research Opportunities Program at the University of Utah, and A.K.S. was supported by NSF grant DMS~2101671.}
\maketitle

\section{Introduction}

Consider a finite group $G$ acting on a noetherian ring $R$ via ring automorphisms. The question whether the invariant ring $R^G$ is noetherian is a classical one, with positive results, in a sense, going back to Hilbert and Noether: If the order of the finite group is invertible in $R$, then~$R^G$ is noetherian~\cite{Hilbert:1890, Hilbert:1893, Noether:1915}; if $R$ is a finitely generated algebra over a noetherian ring $A$, and the action of $G$ on $R$ is via $A$-algebra automorphisms, then, again,~$R^G$ is noetherian~\cite{Noether:1926}.

On the other hand, Nagata gave an example of an artinian local ring $R$ containing a field of characteristic $p>0$, with an action of a cyclic group $G$ of order $p$, such that~$R^G$ is not noetherian~\cite[Proposition~0.10]{Nagata:1969}, see also~\cite[\S1]{Fogarty} and~\cite[Example~12]{Kollar}; while the ring in this example is of course not an integral domain, in the same paper, Nagata also constructs a pseudo-geometric local integral domain $R$ of dimension one and characteristic~$p>0$, with an action of a cyclic group $G$ of order~$p$, such that $R^G$ is not noetherian~\cite[Proposition~0.11]{Nagata:1969}. In contrast, if $G$ is a finite group acting on a Dedekind domain $R$, then $R^G$ is noetherian~\cite[Proposition~0.3, Remark~0.7]{Nagata:1969}.

In light of the above, it is natural to impose stronger hypotheses on $R$ and ask whether~$R^G$ is noetherian when $G$ is a finite group, and $R$ is normal \cite[Question~0.1]{Nagata:1969}, or even regular. These questions were settled in the negative by Nagarajan~\cite[\S4]{Nagarajan}, who constructed a formal power series ring~$R$ of dimension two, over a field of characteristic two, with the action of an involution $\sigma$ such that~$R^{\la\sigma\ra}$ is not noetherian. Our first goal in this paper is to point out how Nagarajan's example readily extends to each positive prime characteristic~$p$, providing an action of a cyclic group~$G$ of order $p$ on a formal power series ring $R\colonequals K\llbracket x,y\rrbracket$, with $K$ a field of characteristic $p$, such that the invariant ring $R^G$ is not noetherian. Other variations of Nagarajan's example may be found in~\cite{CL} and~\cite{Baba}.

Our other goal is to note that Nagarajan's construction extends readily to a curious characteristic zero example: there exists a formal power series ring $R\colonequals K\llbracket x,y\rrbracket$ over a field~$K$ of characteristic zero, with an action of the infinite cyclic group $G$, such that the invariant ring~$R^G$ is not noetherian. The positive characteristic and the characteristic zero examples are all sharp: in each case, the dimension of the regular local ring $R$ is the least possible, see Remark~\ref{remark:dvr}, as is the cardinality and the number of generators of the group~$G$.

While we have focused here on the noetherian property of $R^G$, related questions on the finite generation of $R^G$ have a rich history: in addition to Nagata's celebrated counterexamples to Hilbert's $14$th Problem~\cite{Nagata:1959, Nagata:1960}, we point the reader towards the papers~\cite{Mumford, Roberts, Steinberg, DF, Mukai, Freudenburg, Kuroda, Totaro}, and the references therein.

\section{The example}

Let $\FF$ be a field, and consider the purely transcendental extension field
\[
K\colonequals\FF(a_1,b_1,a_2,b_2,\dots),
\]
where the elements $a_n, b_n$ are indeterminates over $\FF$. Set $R\colonequals K\llbracket x,y\rrbracket$, i.e., $R$ is the ring of formal power series in the variables $x$ and $y$, with coefficients in $K$. Set
\[
f_n\colonequals a_nx+b_ny\qquad\text{ for }n\ge 1.
\]
Define an $\FF$-algebra endomorphism $\sigma$ of $R$ as follows:
\[
\sigma\colon\begin{cases}
x & \mapsto \ \ \ x,\\
y & \mapsto \ \ \ y,\\
a_n & \mapsto \ \ \ a_n+yf_{n+1},\\
b_n & \mapsto \ \ \ b_n-xf_{n+1}.
\end{cases}
\]
It is readily seen that $\sigma(f_n)=f_n$ for each $n\ge 1$, and also that $\sigma$ is surjective, hence an automorphism of $R$. With this notation, we prove:

\begin{theorem}
Let $K$ be a field constructed as above, $R\colonequals K\llbracket x,y\rrbracket$ a formal power series ring, and $G\colonequals\la\sigma\ra$ a cyclic group acting on $R$ as described above. If the field $K$ has positive characteristic $p$, then $G$ is a cyclic group of order $p$, whereas $G$ is infinite if $K$ has characteristic zero. In either case, the ring of invariants $R^G$ is not noetherian.
\end{theorem}

\begin{proof}
For each $k\in\ZZ$, one has
\[
\sigma^k(a_n)\ =\ a_n+kyf_{n+1}\quad\text{ and }\quad\sigma^k(b_n)\ =\ b_n-kxf_{n+1},
\]
so the group $\la\sigma\ra$ has order $p$ if $K$ has characteristic $p>0$, and is infinite cyclic otherwise.

Let $\frakm$ denote the maximal ideal of $R$. We claim that for each $\alpha$ in $K$, one has
\begin{equation}
\label{equation:action:scalars}
\sigma(\alpha)\ \equiv\ \alpha\mod\frakm^2
\end{equation}
in $R$. To see this, suppose $\alpha=g/h$ for nonzero $g,h$ in $\FF[a_1,b_1,a_2,b_2,\dots]$. It is immediate from the definition that $\sigma(g)\equiv g\mod\frakm^2$. Since~$g$ is a unit in $R$, there exists~$g_2\in\frakm^2$ with~$\sigma(g)=g(1-g_2)$. Similarly, there exists $h_2\in\frakm^2$ with $\sigma(h)=h(1-h_2)$. But then
\[
\sigma\left(\frac{g}{h}\right)\ =\ \frac{g(1-g_2)}{h(1-h_2)}\ =\ \frac{g}{h}(1-g_2)(1+h_2+h_2^2+\cdots) \equiv\ \frac{g}{h}\mod\frakm^2,
\]
which proves the claim.

Given a power series $r\in R$, set $\bar{r}$ to be its constant term, i.e., $\bar{r}\in K$, and $r\equiv\bar{r}\mod\frakm$. We next claim that if $r\in R^G$, then~\eqref{equation:action:scalars} can be strengthened to
\begin{equation}
\label{equation:action:constant:term}
\sigma(\bar{r})\ \equiv\ \bar{r}\mod(x^2,\ y^2)R.
\end{equation}
Given $r\in R^G$, let $\alpha,\beta,\gamma$ be elements of $K$ such that
\[
r\ \equiv\ \bar{r}+\alpha x+\beta y +\gamma xy \mod(x^2,\ y^2)R.
\]
Since $\sigma(r)=r$, one has
\[
\sigma(\bar{r})+\sigma(\alpha) x+\sigma(\beta) y +\sigma(\gamma) xy\ \equiv\ \bar{r}+\alpha x+\beta y +\gamma xy \mod(x^2,\ y^2)R.
\]
By~\eqref{equation:action:scalars}, one has $\sigma(\alpha)\equiv\alpha\mod\frakm^2$, and $\sigma(\beta)\equiv\beta\mod\frakm^2$, and $\sigma(\gamma)\equiv\gamma\mod\frakm^2$, so the above display yields $\sigma(\bar{r})\equiv\bar{r}\mod(x^2,y^2)R$ as desired.

Lastly, we prove that $R^G$ is not noetherian by showing that
\[
f_{n+1}\ \notin\ (f_1,\dots,f_n)R^G\qquad\text{ for }n\ge 1,
\]
which, then, gives a strictly ascending chain of ideals in $R^G$. Suppose, to the contrary, that there exists an integer $n$ such that
\[
f_{n+1}\ =\ \sum_{k=1}^nr_kf_k
\]
where $r_k\in R^G$ for each $k$ with $1\le k\le n$. The above may be written as
\[
a_{n+1}x+b_{n+1}y\ =\ \sum_{k=1}^nr_k(a_kx+b_ky),
\]
so comparing the coefficients of $x$ yields
\begin{equation}
\label{equation:a}
a_{n+1}\ =\ \sum_{k=1}^n\bar{r}_k\,a_k.
\end{equation}
Applying $\sigma$ to the above equation gives
\[
a_{n+1}+yf_{n+2}\ =\ \sum_{k=1}^n\sigma(\bar{r}_k)(a_k+yf_{k+1}),
\]
i.e.,
\[
a_{n+1}+a_{n+2}xy+b_{n+2}y^2\ =\ \sum_{k=1}^n\sigma(\bar{r}_k)(a_k+a_{k+1}xy+b_{k+1}y^2).
\]
Since $\sigma(\bar{r}_k)\equiv\bar{r}_k\mod(x^2,y^2)R$ for each $k$ by~\eqref{equation:action:constant:term}, one obtains
\[
a_{n+1}+a_{n+2}xy\ \equiv\ \sum_{k=1}^n\bar{r}_k(a_k+a_{k+1}xy)\mod(x^2,\ y^2)R.
\]
In light of~\eqref{equation:a}, this simplifies to
\[
a_{n+2}xy\ \equiv\ \sum_{k=1}^n\bar{r}_k\,a_{k+1}xy\mod(x^2,\ y^2)R,
\]
from which one obtains
\begin{equation}
\label{equation:b}
a_{n+2}\ =\ \sum_{k=1}^n\bar{r}_k\,a_{k+1}.
\end{equation}
Repeating the argument that~\eqref{equation:a} implies~\eqref{equation:b} gives
\[
a_{n+m+1}\ =\ \sum_{k=1}^n\bar{r}_k\,a_{k+m}\qquad\text{ for each }m\ge 1.
\]
As $\bar{r}_1,\dots,\bar{r}_n$ are finitely many elements of the field $K$, this contradicts the assumption that~$a_1,a_2,\dots$ are infinitely many elements algebraically independent over $\FF$.
\end{proof}

\begin{remark}
\label{remark:dvr}
Consider a discrete valuation ring $R$, with an action of a group $G$. We claim that the invariant ring $R^G$ is either a field or a discrete valuation ring; in particular, $R^G$ is noetherian. To see this, let $v\colon R\smallsetminus\{0\}\onto\ZZ$ be the discrete valuation, and consider its restriction $\bar{v}\colon R^G\smallsetminus\{0\}\to\ZZ$. If the image of this map is $0$, then $R^G$ is a field; otherwise, the image is generated by a positive integer $n$, which yields the discrete valuation
\[
\frac{1}{n}\,\bar{v}\colon R^G\smallsetminus\{0\}\onto\ZZ.
\]
\end{remark}

\section*{Acknowledgments}

The second author is grateful to Bill Heinzer, Kazuhiko Kurano, and Avinash Sathaye for discussions regarding Nagarajan's paper.


\end{document}